\documentclass[11pt,a4paper, reqno]{amsart}
\usepackage{mathrsfs}

\usepackage{amsmath,amsfonts,verbatim}
\usepackage{latexsym}
\usepackage{amssymb,leftidx}
\usepackage{extarrows}
\usepackage{overpic}
\usepackage[dvipsnames]{xcolor}
\usepackage{epsfig}
\usepackage{subfigure}
\usepackage{tikz}
\usetikzlibrary{arrows.meta,decorations.pathreplacing}
\usepackage{enumitem}
\usepackage{mathtools}
\colorlet{RED}{red}

\usepackage{comment}

\usepackage{amssymb}
\usepackage{mathrsfs}
\usepackage{amscd}
\usepackage{bbm}
\usepackage{float}
\usepackage{cite}
\usepackage[pagebackref=true,colorlinks=false]{hyperref}
\hypersetup{
  pdfborder={0 0 0.65},
  linkbordercolor={1 0 0},
  citebordercolor={0 0.65 0},
  urlbordercolor={0 0 1},
  pdftitle={A sharp regularity threshold for Schrodinger maximal estimates on standard tori},
  pdfauthor={Xi Cen, Xitao Gao, Junyong Zhang}
}
\usepackage{microtype}
\newcommand{\arxiv}[1]{\href{https://arxiv.org/abs/#1}{arXiv\,\nolinkurl{#1}}}
\renewcommand*{\backref}[1]{}

\renewcommand*{\backrefalt}[4]{%
  \ifcase #1\relax
  \or\space{\footnotesize #2}
  \else\space{\footnotesize #2}
  \fi}

\newcommand\R{\mathbb{R}}

\newcommand\Z{\mathbb{Z}}
\newcommand\N{\mathbb{N}}
\newcommand{\F}{\mathbb F}

\newcommand{\T}{\mathbb T}
\newcommand{\supp}{\operatorname{supp}}

\newcommand\A{\bf A}

\newcommand{\dd}{\,\mathrm d}

\DeclareRobustCommand{\Figref}[1]{\mbox{\hyperref[#1]{Figure~\ref*{#1}}}}
\DeclareRobustCommand{\Figrange}[2]{\mbox{\hyperref[#1]{Figures~\ref*{#1}--\ref*{#2}}}}
\DeclareRobustCommand{\Extref}[2]{\mbox{\hyperlink{cite.#1}{#2}}}
\DeclareRobustCommand{\Thmref}[1]{\mbox{\hyperref[#1]{Theorem~\ref*{#1}}}}
\DeclareRobustCommand{\Lemref}[1]{\mbox{\hyperref[#1]{Lemma~\ref*{#1}}}}
\DeclareRobustCommand{\Propref}[1]{\mbox{\hyperref[#1]{Proposition~\ref*{#1}}}}
\DeclareRobustCommand{\Corref}[1]{\mbox{\hyperref[#1]{Corollary~\ref*{#1}}}}
\DeclareRobustCommand{\Remref}[1]{\mbox{\hyperref[#1]{Remark~\ref*{#1}}}}
\DeclareRobustCommand{\Secref}[1]{\mbox{\hyperref[#1]{Section~\ref*{#1}}}}
\DeclareRobustCommand{\Subsecref}[1]{\mbox{\hyperref[#1]{Subsection~\ref*{#1}}}}

\usepackage{graphicx}

\numberwithin{equation}{section}
\newtheorem{proposition}{Proposition}[section]

\newtheorem{lemma}{Lemma}[section]
\newtheorem{theorem}{Theorem}[section]
\newtheorem{corollary}{Corollary}[section]
\newtheorem{remark}{Remark}[section]

\begin{document}

\title[Schr\"odinger maximal estimates on standard tori]{\bf A sharp regularity threshold for Schr\"odinger maximal estimates on standard tori}

\author[X. Cen]{Xi Cen}
\address{Xi Cen, School of Science, China University of Mining and Technology-Beijing, Beijing 100083, People's Republic of China}
\email{xicenmath@gmail.com}

\author[X. Gao]{Xitao Gao}
\address{Xitao Gao, Department of Mathematics, Beijing Institute of Technology, Beijing 100081, People's Republic of China}
\email{xitao\_gao@bit.edu.cn}

\author[J. Zhang]{Junyong Zhang}
\address{Junyong Zhang, Department of Mathematics, Beijing Institute of Technology, Beijing 100081, People's Republic of China}
\email{zhang\_junyong@bit.edu.cn}

\begin{abstract}
We disprove almost everywhere convergence of the Schr\"odinger evolution on the standard torus \(\mathbb{T}^d\) for initial data \(f\in H^s(\mathbb{T}^d)\) when \(s<d/(d+2)\), in all dimensions \(d\ge2\). We construct normalized data with frequencies of size $N$ whose evolution attains size $N^{d/(d+2)}$ on a set of uniformly positive measure. The key ingredient in the construction is that the frequencies lie in an affine congruence class, so that at suitable rational times their phases coincide on a family of well-separated spatial points. The result matches the known estimate for $s>d/(d+2)$ with $d\geq2$. By integer dilation and uniform boundedness, we also obtain a single datum in $H^s$ whose evolution is unbounded along a sequence of times tending to zero whenever $s<d/(d+2)$. We also record a logarithmic upper bound in $2D$ at the critical frequency power and a lower bound on shrinking time intervals.
\end{abstract}

\maketitle

\begin{center}
 \begin{minipage}{120mm}
 \end{minipage}
 \end{center}

\section{Introduction}\label{sec:intro}
\subsection{Background and motivations}\label{sub:background}
We study the pointwise convergence of the Schr\"odinger evolution on the standard torus $\T^d\coloneqq(\R/2\pi\Z)^d$ with $d\ge2$. For a trigonometric polynomial $f$, the solution of
\[
 i\partial_tu+\Delta u=0,\qquad u(0,x)=f(x),
\]
is given by
\[
 e^{it\Delta}f(x)=\sum_{n\in\Z^d}\widehat f(n)e^{i(n\cdot x-t|n|^2)},
\]
where the Fourier coefficients are defined by
\[
\widehat f(n)=\frac{1}{(2\pi)^d}\int_{[0,2\pi)^d} f(x)e^{-in\cdot x}\,\dd x,
\qquad n\in\mathbb{Z}^d,
\]
and
\[
n\cdot x=n_1x_1+\cdots+n_dx_d.
\]
The \emph{Carleson convergence problem} asks for the minimal Sobolev regularity $s$ for the initial data $f\in H^s$ that ensures
\[
 \lim_{t\to0^+}e^{it\Delta}f(x)=f(x)\qquad\text{for almost every }x\in\T^d.
\]
The associated maximal estimate must control the supremum in time before integration in space. This order matters, since the time at which the solution is large may depend on the spatial point.

In all dimensions \(d\ge1\), the Carleson problem for the Schr\"odinger equation in Euclidean space is well understood away from the endpoint.
In one dimension, Carleson \cite{Carl} proved that the convergence result holds for \(s\ge 1/4\). Later, Dahlberg--Kenig \cite{D-K} proved that the exponent \(1/4\) is sharp, i.e. the result fails for \(s<1/4\).
For higher dimensions \(d\ge2\), the situation becomes more complicated. See Sj\"olin \cite{Sjolin}, Vega \cite{Vega}, and Lee \cite{Lee}. 
Bourgain \cite{B} proved the necessary condition $s\ge d/[2(d+1)]$. Du--Guth--Li \cite{DGL} in 2D and Du--Zhang \cite{DZ} in higher dimensions proved the matching sufficient condition. In particular, for $s>d/[2(d+1)]$, they proved
\begin{equation}\label{eq:euclidean-upper}
\left\|\sup_{0<t<1}|e^{it\Delta_{\R^d}}f|\right\|_{L^2(B(0,1))}
 \lesssim_{d,s}\|f\|_{H^s(\R^d)}.
\end{equation}
Hence, \begin{equation}\label{eq:euclidean-convergence}
\lim_{t\to0} e^{it\Delta} f(x)=f(x)\quad \text{a.e. } x\in\mathbb{R}^d
\end{equation}
for every \(f\in H^s(\mathbb{R}^d)\). The endpoint case \(s=d/[2(d+1)]\) remains open for \(d\ge2\).
The following \Figref{fig:euclidean} records these results, with the threshold marked in red.

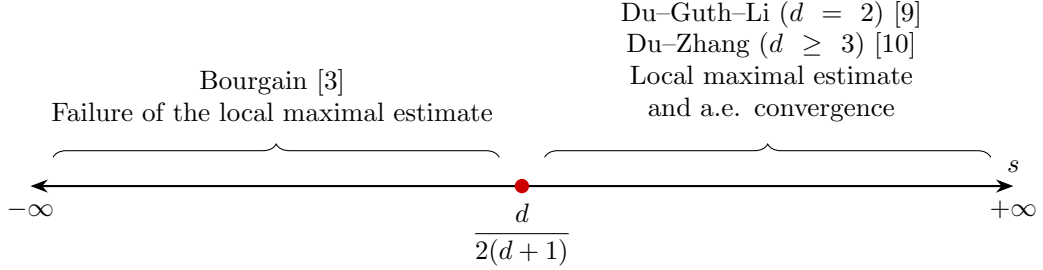
\begin{figure}[H]
\centering
\begin{tikzpicture}[x=1cm,y=1cm,font=\small]
\draw[<->,>=Stealth,thick] (0,0)--(13,0);
\node[below] at (0,-.06) {$-\infty$};
\node[below] at (13,-.06) {$+\infty$};
\node[above] at (13,.05) {$s$};
\draw[decorate,decoration={brace,amplitude=5pt}] (.3,.38)--(6.2,.38);
\draw[decorate,decoration={brace,amplitude=5pt}] (6.8,.38)--(12.7,.38);
\node[align=center,anchor=south,text width=6cm] at (3.2,.72) {Bourgain \cite{B}\\Failure of the local maximal estimate};
\node[align=center,anchor=south,text width=6cm] at (9.8,.72) {Du--Guth--Li $(d=2)$ \cite{DGL}\\Du--Zhang $(d\ge3)$ \cite{DZ}\\Local maximal estimate\\and a.e. convergence};
\fill[red!80!black] (6.5,0) circle (2.7pt);
\node[below,align=center] at (6.5,-.1) {$\displaystyle\frac{d}{2(d+1)}$};
\node[align=center,text width=13cm,font=\footnotesize] at (6.5,-1.45) {Both ranges are strict. The red point marks the threshold, not an endpoint assertion.};
\end{tikzpicture}
\caption{The Euclidean problem on $\R^d$, $d\ge2$. The maximal norm is taken on a fixed spatial ball.}\label{fig:euclidean}
\end{figure}

In contrast, before the present work, the periodic problem remained open in all dimensions. Almost everywhere convergence was known for $s>d/(d+2)$ \cite{CLS,MV,WZ}, and counterexamples were known for $s<d/[2(d+1)]$ \cite{CLS,EL}. Thus the interval $s\in[d/[2(d+1)],d/(d+2)]$ remained unresolved. Partial progress on Weyl sums and level sets was obtained in \cite{Bar,Dem,MYZ}.

On the standard 2D torus, Wang--Zhang \cite[\Extref{WZ}{Theorem~1.1}]{WZ} proved almost everywhere convergence for $s>1/2$. On higher-dimensional tori, Compaan--Luc\`a--Staffilani \cite[\Extref{CLS}{Proposition~3.1}]{CLS} proved
\begin{equation}\label{eq:known-upper}
 \left\|\sup_{0<t<1}|e^{it\Delta}f|\right\|_{L^2(\T^d)}
 \lesssim_{d,s}\|f\|_{H^s(\T^d)}
 \qquad\left(s>\frac{d}{d+2}\right).
\end{equation}
The difference between these two exponents leaves a natural question. Does the periodic problem have the same threshold as the Euclidean problem, or can the arithmetic of the integer frequencies force a larger loss?

The arithmetic of periodic frequencies has also been studied through Weyl sums and level sets. Demeter \cite{Dem} obtained essentially sharp $L^4$ level-set estimates on $\T^1$ in a specified range of levels. This 1D result and the special-coefficient estimates discussed below address different parts of the periodic problem.

Two conjectures give a precise form of the first possibility. Miao--Yuan--Zhao--Barron \cite[p.~3, (1.6) and the following paragraph]{MYZ} proposed the periodic maximal estimate for arbitrary Fourier coefficients with regularity above $d/[2(d+1)]$. In their normalization, the conjectured estimate is
\begin{equation}\label{eq:conjectured}
 \left\|\sup_{0<\tau<1}\left|\sum_{|n|\le N}a_n
 e^{2\pi i(n\cdot y+\tau|n|^2)}\right|\right\|_{L^r((\R/\Z)^d)}
 \lesssim_{s,r}N^s\left(\sum_n|a_n|^2\right)^{1/2}
\end{equation}
for $1\le r\le2$ and $s>d/[2(d+1)]$. Eceizabarrena--Luc\`a \cite[Introduction, (1.1) and (1.4)]{EL} conjectured that their periodic divergence threshold
\[
 s=\frac{d}{2(d+1)}(d+1-\alpha),\qquad 0<\alpha\le d,
\]
is optimal. Taking $\alpha=d$ gives the same proposed threshold for almost everywhere convergence. In particular, both conjectures predict $1/3$ in 2D.

Below this proposed threshold, Compaan--Luc\`a--Staffilani \cite[\Extref{CLS}{Proposition~3.2}]{CLS} had already proved failure of the periodic maximal estimate. Eceizabarrena--Luc\`a \cite[\Extref{EL}{Theorem~1.1}]{EL}, with $\alpha=d$, constructed data whose evolution diverges on a set of positive measure. \Figref{fig:conjecture} shows this proved negative range, the previously open middle interval, and the known positive range above $d/(d+2)$.

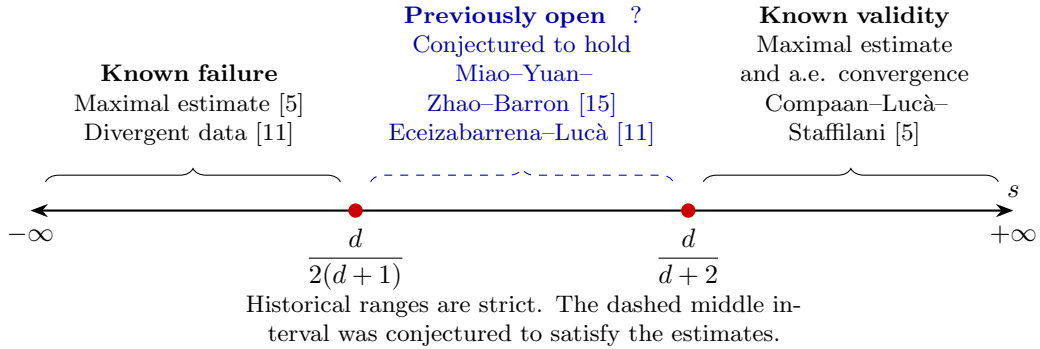
\begin{figure}[H]
\centering
\begin{tikzpicture}[x=1cm,y=1cm,font=\small]
\draw[<->,>=Stealth,thick] (0,0)--(13,0);
\node[below] at (0,-.06) {$-\infty$};
\node[below] at (13,-.06) {$+\infty$};
\node[above] at (13,.05) {$s$};
\draw[decorate,decoration={brace,amplitude=5pt}] (.2,.38)--(4.1,.38);
\draw[decorate,dashed,blue!65!black,decoration={brace,amplitude=5pt}] (4.5,.38)--(8.5,.38);
\draw[decorate,decoration={brace,amplitude=5pt}] (8.9,.38)--(12.8,.38);
\node[align=center,anchor=south,font=\footnotesize,text width=4.1cm] at (2.1,.72) {\textbf{Known failure}\\Maximal estimate \cite{CLS}\\Divergent data \cite{EL}};
\node[align=center,anchor=south,font=\footnotesize,text width=4.1cm,blue!65!black] at (6.5,.72) {\textbf{Previously open}\quad ?\\Conjectured to hold\\Miao--Yuan--Zhao--Barron \cite{MYZ}\\Eceizabarrena--Luc\`a \cite{EL}};
\node[align=center,anchor=south,font=\footnotesize,text width=4.1cm] at (10.9,.72) {\textbf{Known validity}\\Maximal estimate\\and a.e. convergence\\Compaan--Luc\`a--\\Staffilani \cite{CLS}};
\fill[red!80!black] (4.3,0) circle (2.7pt);
\fill[red!80!black] (8.7,0) circle (2.7pt);
\node[below,align=center] at (4.3,-.1) {$\displaystyle\frac{d}{2(d+1)}$};
\node[below,align=center] at (8.7,-.1) {$\displaystyle\frac{d}{d+2}$};
\node[align=center,text width=13cm,font=\footnotesize] at (6.5,-1.45) {Historical ranges are strict. The dashed middle interval was conjectured to satisfy the estimates.};
\end{tikzpicture}
\caption{The periodic problem before this work. The outer intervals were known, while the middle interval was open and conjectured to admit maximal estimates and a.e. convergence. The conjectures are stated in \cite[p.~3, after (1.6)]{MYZ} and \cite[Introduction, (1.4), $\alpha=d$]{EL}. Neither marked equality case is included in the three intervals.}\label{fig:conjecture}
\end{figure}

In this paper, we disprove the conjectured validity on $d/[2(d+1)]<s<d/(d+2)$. We show that the larger exponent in \eqref{eq:known-upper} is necessary on the standard torus. Our examples use a sparse affine set of integer frequencies. At selected rational times, the quadratic phase becomes constant on this set after a spatial translation. We choose enough translations to cover a fixed proportion of the torus by disjoint neighborhoods of radius comparable to $N^{-1}$. On each neighborhood the evolved sum has size $N^{d/(d+2)}$. The positive measure of these sets is also what allows us to obtain a single divergent datum, rather than only a sequence of large maximal norms.


Recent work also considers related geometries and randomized data. Bhimani--Choudhary \cite{BC} proved a frequency-localized maximal estimate on $B^n(0,1)\times\T^m$ and almost everywhere convergence on $\R^n\times\T^m$ for
\[
 s>\frac{n+m}{n+m+2}.
\]
For $n,m\ge1$, their \Extref{BC}{Theorem~1.3} gives counterexamples below $(n+m)/[2(n+m+1)]$. Their \Extref{BC}{Remark~1.4} leaves a gap between these two ranges. Eceizabarrena--Merino \cite{EM} proved almost sure pointwise convergence for the 2D periodic quintic nonlinear Schr\"odinger equation with Gaussian Fourier data of positive regularity. Their statement concerns almost every random choice of the coefficients. Our question concerns every deterministic datum in a Sobolev space, and a counterexample to such a statement is compatible with almost sure convergence under a particular randomization.

\subsection{Main result}\label{sub:results}
We use normalized Haar measure on $\T^d$ and the Sobolev norm
\[
 \|f\|_{H^s(\T^d)}^2\coloneqq
 \sum_{n\in\Z^d}(1+|n|^2)^s|\widehat f(n)|^2.
\]
If $s>d/(d+2)$, the known estimate \eqref{eq:known-upper} and approximation by trigonometric polynomials give
\[
 \lim_{t\to0^+}e^{it\Delta}f(x)=f(x)
 \qquad\text{for almost every }x\in\T^d
 \quad\text{for every }f\in H^s(\T^d).
\]
Our main result shows that this convergence can fail below $d/(d+2)$.

\begin{theorem}[Failure of pointwise convergence]\label{thm:threshold}
Let $d\ge2$ and $s<d/(d+2)$. The maximal operator
\[
 f\longmapsto\sup_{0<t<1}|e^{it\Delta}f|
\]
is not bounded from $H^s(\T^d)$ to $L^{2,\infty}(\T^d)$, even on trigonometric polynomials.
Moreover, there exist $f\in H^s(\T^d)\cap L^2(\T^d)$, a measurable set $E\subseteq\T^d$ of positive measure, and a sequence $t_j\to0^+$ such that
\begin{equation}\label{eq:divergence}
 \limsup_{j\to\infty}|e^{it_j\Delta}f(x)|=\infty
 \qquad\text{for every }x\in E.
\end{equation}
In particular, $e^{it\Delta}f(x)$ does not converge to $f(x)$ as $t\to0^+$ for $x\in E$.
\end{theorem}

\begin{remark}
The assertion about the maximal operator means that no constant $C$ makes
\[
 \left\|\sup_{0<t<1}|e^{it\Delta}f|\right\|_{L^{2,\infty}(\T^d)}
 \le C\|f\|_{H^s(\T^d)}
\]
hold for every trigonometric polynomial $f$. This formulation also applies when $s<0$. The divergent datum in the theorem belongs to $L^2$, and its values at the times $t_j$ are given by the almost everywhere limits of the full Fourier cube truncations. We justify this identification in the proof.
\end{remark}

\begin{remark}
Together with the known convergence for $s>d/(d+2)$, the theorem gives the threshold $d/(d+2)$ away from equality. In 2D it is $1/2$. The estimate and convergence at $s=d/(d+2)$ remain undecided here.
\end{remark}

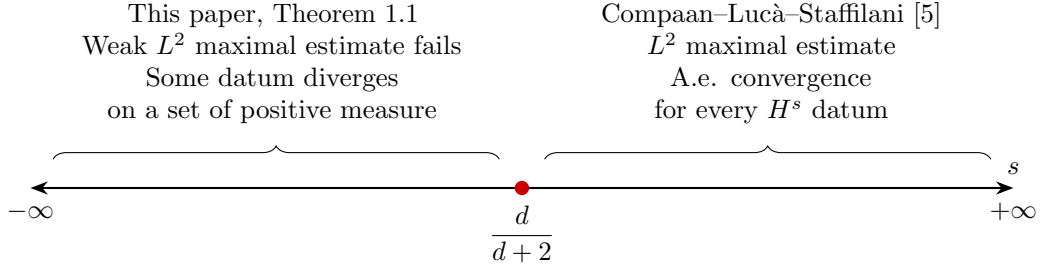
\begin{figure}[H]
\centering
\begin{tikzpicture}[x=1cm,y=1cm,font=\small]
\draw[<->,>=Stealth,thick] (0,0)--(13,0);
\node[below] at (0,-.06) {$-\infty$};
\node[below] at (13,-.06) {$+\infty$};
\node[above] at (13,.05) {$s$};
\draw[decorate,decoration={brace,amplitude=5pt}] (.3,.38)--(6.2,.38);
\draw[decorate,decoration={brace,amplitude=5pt}] (6.8,.38)--(12.7,.38);
\node[align=center,anchor=south,text width=6cm] at (3.2,.72) {This paper, \Thmref{thm:threshold}\\Weak $L^2$ maximal estimate fails\\Some datum diverges\\on a set of positive measure};
\node[align=center,anchor=south,text width=6cm] at (9.8,.72) {Compaan--Luc\`a--Staffilani \cite{CLS}\\$L^2$ maximal estimate\\A.e. convergence\\for every $H^s$ datum};
\fill[red!80!black] (6.5,0) circle (2.7pt);
\node[below,align=center] at (6.5,-.1) {$\displaystyle\frac{d}{d+2}$};
\node[align=center,text width=13cm,font=\footnotesize] at (6.5,-1.45) {Both ranges are strict. The equality case is not decided in this paper.};
\end{tikzpicture}
\caption{The standard torus $\T^d$, $d\ge2$. The known upper bound and the new counterexamples have the same non-endpoint threshold.}\label{fig:torus}
\end{figure}

\begin{remark}
\Figrange{fig:euclidean}{fig:torus} concern Sobolev regularity, rather than a spatial integrability exponent. Above the periodic threshold, the maximal estimate holds and every $H^s$ datum converges almost everywhere. Below it, the weak maximal estimate fails and there exists a datum for which convergence fails on a set of positive measure. The latter statement does not assert divergence for every datum.
\end{remark}

\begin{remark}
Recent work on Strichartz estimates \cite{DFGGL}, combined with Bernstein's inequality in time, yields a logarithm-free \(L^2\) maximal estimate with frequency factor \(N^{1/3}\) in one dimension. However, this result should not be regarded as the sharp threshold for the one-dimensional \(L^2\) case. Our higher-dimensional counterexample cannot be directly reduced to lower dimensions, because the one-dimensional quotient group has no element of order \(p^2\), and the repetition of spatial centers causes the positive-measure counting to fail.
\end{remark}

\subsection{Organization}\label{sub:organization}
We prove \Thmref{thm:threshold} in \Secref{sec:threshold}. We first state the positive-measure lower bound needed for the weak maximal estimate. Integer dilation then moves its finite sets of bad times towards zero. A uniform boundedness lemma gives one initial datum whose evolution is unbounded along those times. We identify the Fourier-sum representative and conclude the failure of pointwise convergence before proving the supporting lemmas.

\Secref{sec:construction} gives the arithmetic construction, first in 2D and then in even and odd dimensions. We prove each separation lemma after using it. The supplementary frequency estimate and shrinking-interval consequence are stated and proved in \Secref{sec:log}. The implication from the known upper estimate to pointwise convergence above $d/(d+2)$ was recalled in \Subsecref{sub:results}.

\subsection{Notation}\label{sub:notation} We close this section by introducing and recalling some notation.
\begin{itemize}
\item All integrals on $\T^d$ use the measure $(2\pi)^{-d}\dd x$. Our Fourier convention gives
\[
 \widehat f(n)=(2\pi)^{-d}\int_{[0,2\pi)^d}f(x)e^{-in\cdot x}\dd x,
 \qquad\text{and}\qquad
 \|f\|_2^2=\sum_{n\in\Z^d}|\widehat f(n)|^2.
\]
\item We write $A\lesssim_d B$ if $A\le C_dB$, with $C_d$ independent of the frequency and the prime, and $A\eqsim_d B$ if both inequalities hold. Additional subscripts indicate further dependence. A norm without a domain is taken on the torus under consideration.
\item We use Euclidean length $|n|=\big(\sum_{j=1}^d|n_j|^2\big)^{1/2}$ for frequency annuli and the maximum norm $|n|_\infty=\max_{1\leq j\leq d}|n_j|$ for Fourier cubes. The latter truncations are
\begin{equation}\label{SL}
 S_L(t)f\coloneqq\sum_{|n|_\infty\le L}\widehat f(n)e^{i(n\cdot x-t|n|^2)}.
\end{equation}
 The weak $L^2$ norm is
\begin{equation}\label{w-L2}
 \|g\|_{L^{2,\infty}(\T^d)}\coloneqq
 \sup_{\lambda>0}\lambda\,|\{x\in\T^d\mid |g(x)|>\lambda\}|^{1/2}.
\end{equation}
\end{itemize}

\section{Proof of \Thmref{thm:threshold}}\label{sec:threshold}

In this section, we prove the main \Thmref{thm:threshold}. The key ingredients are the following two propositions.

The first proposition establishes a lower bound for a set of fixed positive measure, with only finitely many times at each frequency scale. Its proof is postponed to \Secref{sec:construction}.

\begin{proposition}
[A positive-measure lower bound]\label{thm:lower}
Let $d\ge2$. There are constants $c_d,b_d>0$, an unbounded sequence of scales $N$, trigonometric polynomials $F_N$, and nonempty finite sets $B_N\subseteq(0,1/2]$ such that
\begin{gather}\label{FN}
 \supp\widehat F_N\subseteq\{n\in\Z^d\mid N\le|n|\le2N\},\\
 \|F_N\|_2=1,
\end{gather}
and
\begin{equation}\label{eq:lower}
 \left|\left\{x\in\T^d\ \middle|\
 \max_{t\in B_N}|e^{it\Delta}F_N(x)|\ge c_dN^{\frac{d}{d+2}}
 \right\}\right|\ge b_d.
\end{equation}
\end{proposition}

We use the finite time sets from \Propref{thm:lower} to obtain an initial datum whose evolution is unbounded as time tends to zero. We also need our second proposition, concerning uniform boundedness. Its proof is postponed for the moment.

\begin{proposition}[uniform boundedness]\label{prop:baire}
Let $E$ be a Banach space, let $\mu(Y)<\infty$, and let $T_j$ be continuous linear maps from $E$ to $L^0(Y,\mu)$, equipped with convergence in measure. Suppose that $\sup_j|T_jf|<\infty$ almost everywhere for every $f\in E$. Then, for each $\epsilon>0$, there is $\delta>0$ such that
\[
 \|f\|_E<\delta\quad\Longrightarrow\quad
 \mu\{\sup_j|T_jf|>1\}\le\epsilon.
\]
\end{proposition}

Now we prove \Thmref{thm:threshold}.
\begin{proof}[Proof of \Thmref{thm:threshold}]
We first prove that the maximal operator
\[
 f\longmapsto\sup_{0<t<1}|e^{it\Delta}f|
\]
is not bounded from $H^s(\T^d)$ to $L^{2,\infty}(\T^d)$.
By definition \eqref{w-L2}, we apply \Propref{thm:lower} directly to obtain 
\begin{align*}
\left\|\sup_{0<t<1}|e^{it\Delta}F_N|\right\|_{L^{2,\infty}}
 &\ge\frac{c_dN^{\frac{d}{d+2}}}2
 \left|\left\{x\in\T^d \middle|\ \max_{t\in B_N}|e^{it\Delta}F_N(x)|>
 \frac{c_dN^{\frac{d}{d+2}}}2\right\}\right|^{1/2} \gtrsim_dN^{\frac{d}{d+2}}.
\end{align*}
By \eqref{FN}, annular localization gives
$\|F_N\|_{H^s}\eqsim_s N^s$. Thus, if $s<d/(d+2)$,
\begin{align*}
 \frac{N^{\frac{d}{d+2}}}{\|F_N\|_{H^s}}
 &\eqsim_sN^{\frac{d}{d+2}-s}\to\infty
\end{align*} as $N\to \infty$.
Thus the maximal estimate fails below $\frac{d}{d+2}$, even with a weak $L^2$ norm on the left. \vspace{0.2cm}

Next we prove \eqref{eq:divergence}. To make the times tend to zero, we use the integer covering maps of the torus. These preserve normalized Haar measure and satisfy, for $K\in\N$, $K\ge1$ and $s\ge0$,
\begin{equation}\label{eq:dilation}
 \begin{aligned}
 \|F(K\cdot)\|_2&=\|F\|_2,\\
 e^{it\Delta}[F(K\cdot)](x)&=(e^{iK^2t\Delta}F)(Kx),\\
 \|F(K\cdot)\|_{H^s}^2
 &=\sum_n(1+K^2|n|^2)^s|\widehat F(n)|^2
 \le K^{2s}\|F\|_{H^s}^2.
 \end{aligned}
\end{equation}
Fix $0<s<\frac{d}{d+2}$. We choose $K_m=m$, and then choose $N_m$ from the sequence in \Propref{thm:lower} sufficiently large. Define
\begin{align*}
 g_m(x)\coloneqq\frac{2}{c_dN_m^{\frac{d}{d+2}}}F_{N_m}(K_mx),\quad
 B'_m\coloneqq K_m^{-2}B_{N_m}.
\end{align*}
Then we obtain
\begin{align*}
 \|g_m\|_{H^s}
 &\le C_{d,s}K_m^sN_m^{s-\frac{d}{d+2}}\le2^{-m},\\
 B'_m&\subseteq(0,m^{-2}),
\end{align*}
and
\begin{align*}
 \left|\left\{x\in\T^d\ \middle|\ \max_{t\in B'_m}
 |e^{it\Delta}g_m(x)|>1\right\}\right|&\ge b_d.
\end{align*}
Enumerating the nonempty finite blocks $B'_m$ consecutively gives a sequence $t_j\to0^+$. Each $T_j=e^{it_j\Delta}$ maps $H^s$ continuously to $L^2$ and hence to $L^0$ (the space topologized by convergence in measure, since $L^2(\mathbb T^d)\hookrightarrow L^0(\mathbb T^d)$). If every datum had almost everywhere bounded evolution along this sequence, \Propref{prop:baire} with $\epsilon=b_d/2$ would contradict the preceding lower bound for all sufficiently large $m$. We conclude that some $f\in H^s$ satisfies
\[
 \left|\left\{x\in\T^d\ \middle|\ \sup_j|T_jf(x)|=\infty\right\}\right|>0.
\]
Each $T_jf$ is finite almost everywhere. Outside the union of these countably many null sets, every finite prefix is bounded. Hence
\[
 \limsup_{j\to\infty}|T_jf(x)|=\infty
\]
 on a set of positive measure.

We identify these time slices with the Fourier-sum representative in Eceizabarrena--Luc\`a \cite{EL}. Since $0<s<\frac{d}{d+2}<d/2$, their formula (2.1) and \Extref{EL}{Proposition~A.1}, with $\alpha=d$, apply at each fixed $t_j$. Together with Parseval, they give
\begin{align*}
 S_L(t_j)f(x)&\longrightarrow V_j(x)
 \qquad\text{for almost every }x,\\
 \|S_L(t_j)f-T_jf\|_2^2
 &=\sum_{|n|_\infty>L}|\widehat f(n)|^2\longrightarrow0,\\
 V_j(x)&=T_jf(x)\qquad\text{for almost every }x.
\end{align*}
Taking the countable union of the exceptional sets makes this identification simultaneous for every $j$. We remove this null set and the null set where $f$ is not finite from the positive-measure set of divergence, and call the remaining set $E$. Then \eqref{eq:divergence} holds at every point of $E$, and convergence to $f(x)$ is impossible there. Finally, if $s\le0$, we use the datum constructed at any fixed $\sigma\in(0,\frac{d}{d+2})$ and the inclusion $H^\sigma\subseteq H^s\cap L^2$.
\end{proof}

To complete the proof of our main \Thmref{thm:threshold},
it remains to prove \Propref{prop:baire} and \Propref{thm:lower}. We first prove \Propref{prop:baire}.
\begin{proof}[Proof of \Propref{prop:baire}]
Fix $\epsilon>0$. For each integer $a\ge1$, define
\[
 A_a\coloneqq\{f\in E\mid\mu\{\sup_j|T_jf|>a\}\le\epsilon/2\}.
\]
Let $f_n\in A_a$ converge to $f$ in $E$. By continuity, we can pass to one subsequence along which $T_jf_n\to T_jf$ almost everywhere for every $j$. The strict level sets then satisfy
\begin{align*}
 \boldsymbol1_{\{\sup_j|T_jf|>a\}}
 &\le\liminf_{n\to\infty}
 \boldsymbol1_{\{\sup_j|T_jf_n|>a\}}.
\end{align*}
By Fatou's lemma, we have
\begin{align*}
 \mu\{\sup_j|T_jf|>a\}
 &\le\liminf_{n\to\infty}\mu\{\sup_j|T_jf_n|>a\}
 \le\epsilon/2.
\end{align*}
Thus $f\in A_a$. Hence $A_a$ is closed. The assumed pointwise boundedness and finiteness of $\mu(Y)$ give $E=\bigcup_{a\ge1}A_a$. Since $E$ is a Banach space (complete) and $A_a$ is closed, Baire's theorem provides a ball $B_E(f_0,r)\subseteq A_a$ for some $a$. For $\|h\|_E<r$, linearity yields
\begin{align*}
 \sup_j|T_jh|
 &\le\sup_j|T_j(f_0+h)|+\sup_j|T_jf_0|,\\
 \mu\{\sup_j|T_jh|>2a\}
 &\le\mu\{\sup_j|T_j(f_0+h)|>a\}
   +\mu\{\sup_j|T_jf_0|>a\}
 \le\epsilon.
\end{align*}
Let $f=h/(2a)$. If $\|f\|_E<\delta$, where $\delta=r/(2a)$, then $\|h\|_E<r$ and $$\sup_j|T_jf|=\frac1{2a}\sup_j|T_jh|\implies \{\sup_j|T_jf|>1\}=\{\sup_j|T_jh|>2a\}.$$
Therefore
\begin{align*}
 \|f\|_E<\frac{r}{2a}
 &\Longrightarrow\mu\{\sup_j|T_jf|>1\}\le\epsilon.\qedhere
\end{align*}
\end{proof}

\begin{remark}[Comparison with the conjectured estimate]\label{rem:normalization}
The lower bound also applies in the normalization of \eqref{eq:conjectured}. For $a_n=\overline{\widehat F_N(n)}$, we have
\[
 \sum_na_ne^{2\pi i(n\cdot y+\tau|n|^2)}
 =\overline{e^{i2\pi\tau\Delta}F_N(-2\pi y)}.
\]
The times $\tau=t/(2\pi)$ with $t\in B_N$ belong to $(0,1)$, and the map $y\mapsto-2\pi y$ preserves normalized measure. Therefore, for $1\le r\le2$,
\[
 \left\|\sup_{0<\tau<1}\left|\sum_na_ne^{2\pi i(n\cdot y+\tau|n|^2)}\right|\right\|_r
 \ge c_db_d^{1/r}N^{\frac{d}{d+2}}.
\]
We choose coefficients with $\ell^2$ norm one that vanish for $|n|>2N$. Replacing $N$ by $2N$ in the conjectured upper bound gives a contradiction whenever $d/[2(d+1)]<s<\frac{d}{d+2}$.
\end{remark}

\section{Proof of \Propref{thm:lower}}\label{sec:construction}
We prove \Propref{thm:lower} by constructing frequencies whose quadratic phases agree at rational times. We first give the 2D example, where both the frequencies and the spatial centers can be written explicitly. We then arrange the same agreement in higher dimensions while keeping the centers separated.

We write $\F_p\coloneqq\Z/p\Z$ for a prime $p$. For an odd integer $q$, let $$e_q(a)\coloneqq e^{2\pi ia/q}.$$Then $e_q(a+q)=e_q(a)$. Orthogonality in $(\Z/q\Z)^d$ is always with respect to the dot product modulo $q$. Thus
\[
 V^\perp\coloneqq
 \{z\in(\Z/q\Z)^d: z\cdot v=0\pmod q\ \text{for every }v\in V\}.
\]
For a residue vector $v$, the expression $|v|^2$ modulo $q$ means $\sum_{r=1}^d v_r^2.$ Real lengths refer only to integer or real vectors. We choose representatives in $[-(q-1)/2,(q-1)/2]^d$ when measuring spatial separation, and take distances on the torus modulo $2\pi\Z^d$.

\subsection{The 2D example}\label{sub:two}
\begin{proof}[Proof of \Propref{thm:lower} in 2D]
We choose a sufficiently large prime $p\equiv1\pmod4$. Let $j\in\F_p$ satisfy $j^2=-1$, and write $v=(j,1)$ and $r=(1,0)$.
For $s\in\F_p$, let $n_s$ be the balanced integer representative of $$r+sv=(1,0)+s(j,1)=(1+sj,s),$$
that is, $$\Z^2\ni n_s\equiv (1+sj,s),\quad (\text{mod}\, p) .$$ Each coordinate belongs to $[-(p-1)/2, (p-1)/2]$. The second coordinate provides enough frequencies away from zero, since
\begin{align*}
 |n_s|&\le p/\sqrt2,\\
 |r+sv|^2&=1+2js+(1+j^2)s^2=1+2js\quad\text{in }\F_p,
\end{align*}
and
\begin{align*}
 \#\{s\in\F_p\mid |n_s|\ge p/4\}
 &\ge\#\{s\in\F_p\mid |(n_s)_2|\ge p/4\}
 =\frac{p-1}{2}\ge\frac p3.
\end{align*}
We split this set (the frequency set $|n_s|\ge p/4$) between the annuli with inner radii $p/4$ and $p/2$. One annulus contains $n_s$ for at least $p/6$ values of $s$. Let $S\subseteq \F_p$ be the set of these parameters, $M=\#S\ge p/6$, and $N$ the inner radius, so $N \eqsim p$. We define
\begin{align*}
 F_N(x)&\coloneqq M^{-1/2}\sum_{s\in S}e^{in_s\cdot x},\\
 \|F_N\|_2^2&=M^{-1}\sum_{s\in S}1=1.
\end{align*}
Its Fourier support lies in $N\le|n|\le2N$ since $N\le|n_s|\le2N$.

We now choose the time at each grid point. The key point is that the time need not be the same for all spatial points. For $\ell=(\ell_1,\ell_2)\in\F_p^2$, let $x_\ell=2\pi\ell/p$. For $a\in\F_p$, represented in $\{0,\ldots,p-1\}$, define $t_a=2\pi a/p$. The phase is
\begin{align*}
 \ell\cdot n_s-a|n_s|^2
 &\equiv\ell_1-a+s(j\ell_1+\ell_2-2aj)\pmod p,\\
 a(\ell)&\coloneqq(2j)^{-1}(j\ell_1+\ell_2),
\end{align*}
which gives
\begin{align*}
 e^{it_{a(\ell)}\Delta}F_N(x_\ell)
 &=M^{-1/2}\sum_{s\in S}e_p(\ell_1-a(\ell))
 =M^{1/2}e_p(\ell_1-a(\ell)).
\end{align*}
To keep the times in $(0,1/2]$, we retain $a\in A_p\coloneqq\{1,\ldots,\lfloor p/(4\pi)\rfloor\}$ and let $B_N=\{2\pi a/p\mid a\in A_p\}$. Then $\#B_N=\#A_p\eqsim p\eqsim N$. Each value of $a(\ell)$ has exactly $p$ preimages, and hence
\[
 \#\Gamma_p=p\#A_p\eqsim p^2,
 \qquad\text{where}\qquad
 \Gamma_p\coloneqq\{\ell\in\F_p^2\mid a(\ell)\in A_p\},
\]
where $\Gamma_p$ contains the ``good'' grid points for which $t_{a(\ell)}\in B_N$. This shows that the good lattice points are not sparse. They constitute a fixed positive proportion of all $p^2$ lattice points.
We enlarge the grid points to small squares. Fix $0<\eta<1/10$ and let $Q_\ell=x_\ell+[-\eta/p,\eta/p]^2$. These squares are disjoint on $\T^2$. For $|h|_\infty\le\eta/p$, one has $x_\ell+h\in Q_{\ell}$. We now verify that after moving from the center point \(x_\ell\) to \(x_\ell + h\), although the Fourier waves are no longer completely in phase, the phase deviation is still very small. The preceding identity gives $|n_s\cdot h|\le\eta$ and
\begin{align*}
 |e^{it_{a(\ell)}\Delta}F_N(x_\ell+h)|
 &=M^{-1/2}\left|\sum_{s\in S}e^{in_s\cdot h}\right|\\
 &\ge M^{-1/2}\sum_{s\in S}\cos(n_s\cdot h)
 \ge (\cos\eta)M^{1/2}\gtrsim N^{1/2},
\end{align*} and the total normalized measure is
\begin{align*}
 \left|\bigcup_{\ell\in\Gamma_p}Q_\ell\right|
 &=\#\Gamma_p\left(\frac{\eta}{\pi p}\right)^2\gtrsim1.
\end{align*}
The primes $p\equiv1\pmod4$ are unbounded, and $N\eqsim p$. This proves the required assertion. In fact, by combining the above results, for infinitely many \(N\), we have constructed a function \(F_N\) and a time set \(B_N\) satisfying
\begin{gather*}
\|F_N\|_{L^2(\mathbb{T}^2)} = 1, \\
\operatorname{supp} \widehat{F}_N \subseteq \{n \in \mathbb{Z}^2 : N \le |n| \le 2N\},
\end{gather*}
and there exists a set \(E_N \subseteq \mathbb{T}^2\) of uniformly positive measure such that for every \(x \in E_N\), there exists \(t \in B_N \subset (0, 1/2]\) satisfying
\[
|e^{it\Delta} F_N(x)| \gtrsim N^{1/2}.
\]
Equivalently,
\[
\left| \left\{ x \in \mathbb{T}^2 : \sup_{t \in B_N} |e^{it\Delta} F_N(x)| \gtrsim N^{1/2} \right\} \right| \gtrsim 1.
\]
\end{proof}

\subsection{The construction in even dimensions}\label{sec:even}
Let $d=2m\ge4$. The affine phase calculation still gives many spatial centers. We need residue classes whose distinct centers remain separated at scale $N^{-1}$. The following lemma makes this choice.

\begin{lemma}[Separation modulo a prime]\label{lem:even}
For every sufficiently large prime $p\equiv1\pmod4$, there are subspaces $V\subseteq U\subseteq\F_p^{2m}$ and $r_0\in U\setminus V$ such that
\[
 V=V^\perp,\qquad \dim V=m,\qquad\text{and}\qquad U=V+\F_pr_0.
\]
The congruence lattice
\[
 \Lambda_V\coloneqq\{n\in\Z^{2m}\mid n\bmod p\in V\}
\]
has an orthogonal basis with every vector of length $\sqrt p$. There is a constant $\gamma_m>0$, independent of $p$, such that, for
\[
 N\coloneqq\lfloor p^{(m+1)/(2m)}\rfloor
 \qquad\text{and}\qquad R\coloneqq\gamma_m p/N,
\]
every nonzero $h\in\Z^{2m}$ with $h\bmod p\in U$ satisfies $|h|_\infty>R$.
\end{lemma}

\begin{proof}[Proof of \Propref{thm:lower} in even dimensions]
We take $V,U,r_0$ from \Lemref{lem:even}. Fix $0<\delta<1/(100d)$, depending only on $d$, and let
\begin{equation}\label{eq:box}
 \Omega_N\coloneqq[N,(1+\delta)N]\times[-\delta N,\delta N]^{d-1}
 \subseteq\{\xi\in\R^d\mid N\le|\xi|\le2N\}.
\end{equation}
For an integer lift $\widetilde r_0\in \Z^{2m}$ of $r_0$, let $$S_N=\Omega_N\cap(\widetilde r_0+\Lambda_V),\quad M_N=\#S_N.$$ Changing the lift does not change the affine lattice since
$p\Z^{2m}\subseteq\Lambda_V$. Choose a half-open fundamental cell containing zero, of diameter $\rho_p\lesssim_d\sqrt p$. Let $\Omega_N^+$ be the $\rho_p$-neighborhood of $\Omega_N$ and $\Omega_N^-$ its interior at distance greater than $\rho_p$ from the complement. The cells based at $S_N$ cover $\Omega_N^-$ and lie in $\Omega_N^+$. Their interiors are disjoint. We cover the boundary layer by $2d$ slabs of thickness $O_d(\sqrt p)$ and remaining side lengths $O_d(N)$. Thus
\begin{align*}
 \operatorname{vol}(\Omega_N^-)&\le M_Np^m\le\operatorname{vol}(\Omega_N^+),\\
 \operatorname{vol}(\Omega_N^+\setminus\Omega_N^-)
 &\lesssim_d N^{d-1}\sqrt p,
\end{align*}
and
\begin{align*}
 \left|M_N-\frac{\operatorname{vol}(\Omega_N)}{p^m}\right|
 &\lesssim_d\frac{N^{d-1}\sqrt p}{p^m},\\
 M_N&\eqsim_d\frac{N^{2m}}{p^m}\eqsim_dp.
\end{align*}
Here $\delta$ is fixed, $\operatorname{vol}(\Omega_N)\eqsim_d N^{2m}$ and $\sqrt p/N\to0$. The count is uniform in the affine translate. We now define
\begin{align*}
 F_N(x)&\coloneqq M_N^{-1/2}\sum_{n\in S_N}e^{in\cdot x},\\
 \|F_N\|_2^2&=M_N^{-1}\sum_{n\in S_N}1=1.
\end{align*}
For $a\in A_p\coloneqq\{1,\ldots,\lfloor p/(4\pi)\rfloor\}$ and $u\in V$, we choose
\[
 \ell\coloneqq2ar_0+u,\qquad x_\ell\coloneqq2\pi\ell/p,
 \qquad\text{and}\qquad t_a\coloneqq2\pi a/p.
\]
Every frequency in $S_N$ has the form $n\equiv r_0+v\pmod p$ with $v\in V$. Since $u\cdot v=v\cdot v=0$ modulo $p$, we obtain
\begin{align}
 \ell\cdot n-a|n|^2
 &\equiv(2ar_0+u)\cdot(r_0+v)-a(|r_0|^2+2r_0\cdot v)\notag\\
 &\equiv a|r_0|^2+u\cdot r_0\pmod p,\label{eq:phase}\\
 e^{it_a\Delta}F_N(x_\ell)
 &=M_N^{-1/2}\sum_{n\in S_N}e_p(a|r_0|^2+u\cdot r_0)
 =M_N^{1/2}e_p(a|r_0|^2+u\cdot r_0).\notag
\end{align}

It remains to count and separate these centers. Let $B_N=\{t_a\mid a\in A_p\}\subseteq(0,1/2]$ and $\Gamma_p=\{2ar_0+u\mid a\in A_p,\ u\in V\}$. As $r_0\notin V$, the cosets $2ar_0+V$ are distinct for $a\in A_p$. We obtain
\[
 \#\Gamma_p=\#A_p\,p^m\eqsim p^{m+1}\eqsim N^d.
\]
For $\ell\ne\ell'$, let $h$ be the balanced representative of $\ell-\ell'$. We have $h\ne0$ and $h\bmod p\in U$. Applying \Lemref{lem:even} gives
\[
 \operatorname{dist}_\infty(x_\ell,x_{\ell'})
 =\frac{2\pi}{p}|h|_\infty>\frac{2\pi R}{p}
 =\frac{2\pi\gamma_m}{N}.
\]
We choose $0<\eta_d<\min\{\pi\gamma_m,1/(10\sqrt d)\}$. The boxes $Q_\ell=x_\ell+[-\eta_d/N,\eta_d/N]^d$ are disjoint. For $|h|_\infty\le\eta_d/N$, as before, \eqref{eq:phase} yields
\begin{align*}
 |n\cdot h|&\le2\sqrt d\,\eta_d\qquad(n\in S_N),\\
 |e^{it_a\Delta}F_N(x_\ell+h)|
 &\ge M_N^{-1/2}\sum_{n\in S_N}\cos(n\cdot h)\ge\tfrac12M_N^{1/2}\gtrsim_dp^{1/2}\eqsim N^{d/(d+2)},\\
 \left|\bigcup_{\ell\in\Gamma_p}Q_\ell\right|
 &=\#\Gamma_p\left(\frac{\eta_d}{\pi N}\right)^d\gtrsim_d1.
\end{align*}
This proves the lower bound in even dimensions.
\end{proof}

Now we are left to prove \Lemref{lem:even}.
\begin{proof}[Proof of \Lemref{lem:even}]
By the two-square theorem, we write $p=a^2+b^2$ and take $j=ab^{-1}\in\F_p$. Let
\[
 \Phi(s)\coloneqq(js_1,s_1,\ldots,js_m,s_m)
 \qquad\text{and}\qquad V\coloneqq\Phi(\F_p^m).
\]
The identity $j^2=-1$ gives
\begin{align*}
 \Phi(s)\cdot\Phi(t)&=(j^2+1)\sum_{r=1}^ms_rt_r=0,\\
 \dim V&=m,\quad
 V^\perp=V.
\end{align*}
For one coordinate pair, define
\[
 \Lambda_0\coloneqq\{(x,y)\in\Z^2\mid bx-ay\equiv0\pmod p\}.
\]
The vectors $g_1=(a,b)$ and $g_2=(-b,a)$ lie in $\Lambda_0$, and
\begin{align*}
 g_1\cdot g_2&=0,\\
 |g_1|^2=|g_2|^2&=p,\\
 |\det(g_1,g_2)|&=p=[\Z^2:\Lambda_0],\\
 \Lambda_V=\Lambda_0^m&=(\Z g_1+\Z g_2)^m.
\end{align*}
This proves the assertion about the lattice basis.

We choose $U$ by excluding subspaces containing a short integer vector. For nonzero $c\in\F_p^m$, let $W_c=\{\Phi(s)\mid c\cdot s=0\}$. This subspace depends only on the projective class $[c]$. With $\lambda(h)=(jh_1+h_2,\ldots,jh_{2m-1}+h_{2m})$, every $0<|h|_\infty\le R$ satisfies
\begin{align*}
 0<|h|^2&\le2mR^2<p,\\
 h\bmod p&\notin V,\\
 \lambda(h)&\ne0,
\end{align*}
for sufficiently large $p$. Indeed, membership in $V$ would force $|h|^2\equiv0\pmod p$. It follows that
\begin{align*}
 h\bmod p\in W_c^\perp
 &\Longleftrightarrow \lambda(h)\cdot s=0\ \text{whenever }c\cdot s=0\\
 &\Longleftrightarrow [c]=[\lambda(h)].
\end{align*}
Thus each short vector excludes only one projective class. We first choose $\gamma_m>0$ small enough and then take $p$ large. Since $R\to\infty$,
\begin{align*}
 \#\mathbb P^{m-1}(\F_p)&=\frac{p^m-1}{p-1}\ge p^{m-1},\\
 \#\{h\in\Z^{2m}\mid0<|h|_\infty\le R\}
 &\le(2R+1)^{2m}
 \le C_m\gamma_m^{2m}p^{m-1}<\tfrac12p^{m-1}.
\end{align*}
We can choose $[c]$ for which $W_c^\perp$ contains none of these vectors. Taking $U=W_c^\perp$, we have
\[
 V\subseteq U,\qquad \dim U=m+1,
 \qquad\text{and}\qquad U=V+\F_pr_0
\]
for any $r_0\in U\setminus V$, as required.
\end{proof}

\subsection{The construction in odd dimensions}\label{sec:odd}
Let $d=2m+1\ge3$. We use the modulus $p^2$ to accommodate the remaining coordinate. The form of the construction is unchanged, but the separation argument must distinguish elements of order $p$ from those of order $p^2$.

\begin{lemma}[Separation modulo a prime square]\label{lem:odd}
For every sufficiently large prime $p\equiv1\pmod4$, let $q=p^2$. There are subgroups $V\subseteq U\subseteq(\Z/q\Z)^d$ and $r_0\in U$ such that
\begin{gather*}
 V=V^\perp,\qquad |V|=p^d,\qquad v\cdot w=0\quad(v,w\in V),\\
 U/V=\langle r_0+V\rangle\cong\Z/q\Z.
\end{gather*}
The lattice $\Lambda_V\coloneqq\{n\in\Z^d\mid n\bmod q\in V\}$ has an orthogonal basis with every vector of length $p$. There is a constant $\gamma_d>0$, independent of $p$, such that, for
\[
 N\coloneqq\lfloor p^{(d+2)/d}\rfloor
 \qquad\text{and}\qquad R\coloneqq\gamma_dq/N,
\]
every nonzero $h\in\Z^d$ with $h\bmod q\in U$ satisfies $|h|_\infty>R$.
\end{lemma}

\begin{proof}[Proof of \Propref{thm:lower} in odd dimensions]
We take $V,U,r_0$ from \Lemref{lem:odd} and retain the box $\Omega_N$ in \eqref{eq:box}. For an integer lift $\widetilde r_0$ of $r_0$, let $S_N=\Omega_N\cap(\widetilde r_0+\Lambda_V)$ and $M_N=\#S_N$.
The fundamental cells have volume $p^d$ and diameter $O_d(p)$. The same inner and outer cell comparison as above gives
\begin{align*}
 M_N&=\frac{\operatorname{vol}(\Omega_N)}{p^d}
 +O_d\left(\frac{N^{d-1}p}{p^d}\right)\eqsim_dp^2=q,
\end{align*}
where $p/N\to0$. Define the polynomial
\begin{align*}
 F_N(x)&\coloneqq M_N^{-1/2}\sum_{n\in S_N}e^{in\cdot x},\\
 \|F_N\|_2^2&=M_N^{-1}\sum_{n\in S_N}1=1.
\end{align*}
Its Fourier support lies in $N\le|n|\le2N$. We take $A_q=\{1,\ldots,\lfloor q/(4\pi)\rfloor\}$, the times $B_N=\{2\pi a/q\mid a\in A_q\}\subseteq(0,1/2]$, and the centers $\Gamma_q=\{2ar_0+u\mid a\in A_q,\ u\in V\}$.
For $\ell=2ar_0+u$ with $a\in A_q$ and $u\in V$, let $x_\ell=2\pi\ell/q$ and $t_a=2\pi a/q$. The phase identity \eqref{eq:phase} uses only $u\cdot v=v\cdot v=0$, and hence also holds modulo $q$. Moreover,
\begin{align*}
 2ar_0+V=2a'r_0+V
 &\Longrightarrow 2(a-a')(r_0+V)=0
 \Longrightarrow a=a'\pmod q,\\
 \#\Gamma_q&=\#A_q\,|V|\eqsim p^{d+2}\eqsim N^d,\\
 |e^{it_a\Delta}F_N(x_\ell)|&=M_N^{1/2}\eqsim_dp.
\end{align*}
Here we used the order $q$ of $r_0+V$ and the invertibility of $2$ modulo $q$.

For distinct centers, their balanced difference is nonzero and belongs to $U$ modulo $q$. By \Lemref{lem:odd},
\[
 \operatorname{dist}_\infty(x_\ell,x_{\ell'})
 >2\pi R/q=2\pi\gamma_d/N.
\]
Fix $0<\eta_d<\min\{\pi\gamma_d,1/(10\sqrt d)\}$. The boxes $Q_\ell=x_\ell+[-\eta_d/N,\eta_d/N]^d$ are disjoint. For $|h|_\infty\le\eta_d/N$, we conclude that
\begin{align*}
 |n\cdot h|&\le2\sqrt d\,\eta_d\qquad(n\in S_N),\\
 |e^{it_a\Delta}F_N(x_\ell+h)|
 &\ge M_N^{-1/2}\sum_{n\in S_N}\cos(n\cdot h)\ge\tfrac12M_N^{1/2}\gtrsim_dp\eqsim N^{d/(d+2)},\\
 \left|\bigcup_{\ell\in\Gamma_q}Q_\ell\right|
 &=\#\Gamma_q\left(\frac{\eta_d}{\pi N}\right)^d\gtrsim_d1.
\end{align*}
This completes the construction.
\end{proof}

Now we prove \Lemref{lem:odd}.
\begin{proof}[Proof of \Lemref{lem:odd}]
By the two-square theorem, we write $p=\rho^2+\sigma^2$ and choose $A=\rho^2-\sigma^2$ and $B=2\rho\sigma$. Then $A^2+B^2=p^2=q$ and $p\nmid B$. With $j=AB^{-1}\in\Z/q\Z$, we have $j^2=-1$. We take $L=\{(js,s)\mid s\in\Z/q\Z\}$ and $P=p\Z/q\Z$, and define $V=L^m\times P$.
Direct calculation gives
\begin{align*}
 (x,y)\in L^\perp
 &\Longleftrightarrow jx+y=0\Longleftrightarrow(x,y)\in L,\\
 z\in P^\perp
 &\Longleftrightarrow pz=0\pmod {p^2}\Longleftrightarrow z\in P,
\end{align*}
and
\begin{align*}
 V^\perp&=V, \quad
 v\cdot w=0\quad(v,w\in V),\\
 |V|&=q^mp=p^d.
\end{align*}
For each coordinate pair, $(A,B)$ and $(-B,A)$ form an orthogonal basis of the congruence lattice. Their determinant and the lattice index both equal $q$. Together with the last-coordinate vector of length $p$, they give the asserted basis of $\Lambda_V$.

We next choose a cyclic subgroup in the quotient by $V$. Consider the surjective homomorphism
\[
 (\Z/q\Z)^d\xrightarrow{\ \pi\ }H\coloneqq(\Z/q\Z)^m\times\F_p,
\]
where
\[
 \pi(z)\coloneqq(jz_1+z_2,\ldots,jz_{2m-1}+z_{2m},z_d\bmod p).
\]
Its kernel is $V$. Let $\mathcal C$ be the family of cyclic subgroups of $H$ of order $p^2$. An element $(c,b)\in H$ has order $p^2$ exactly when $c\notin p(\Z/q\Z)^m$. Each such subgroup has $p(p-1)$ generators, whereas $H$ has $p(p^{2m}-p^m)$ elements of order $p^2$. Hence
\begin{equation}\label{eq:cyclic}
 \#\mathcal C
 =\frac{p(p^{2m}-p^m)}{p(p-1)}
 =\frac{p^{2m}-p^m}{p-1}\eqsim p^{2m-1}=p^{d-2}.
\end{equation}

We exclude any $C\in\mathcal C$ containing $\pi(h)$ for a vector $0<|h|_\infty\le R$. First, $\pi(h)\ne0$. Otherwise, the lattice basis just obtained would give
\[
 p\le|h|\le\sqrt d\,R<p,
\]
which is impossible for large $p$, since $R/p\to0$. If $\pi(h)$ has order $p^2$, it generates the only member of $\mathcal C$ containing it. For an element of order $p$, write $\pi(h)=(pa,b)$ with $a\in\F_p^m$ and $b\in\F_p$. We have
\begin{align*}
 \#\{C\in\mathcal C\mid(pa,b)\in C\}&=0\qquad(b\ne0),\\
 \#\{C\in\mathcal C\mid(pa,0)\in C\}
 &=\frac{(p-1)p^m p}{p(p-1)}=p^m\qquad(a\ne0).
\end{align*}
The first identity follows from the zero last coordinate of every order-$p$ element of a cyclic group of order $p^2$. For the second, a generator has first component reducing to a nonzero multiple of $a$ modulo $p$. There are $p-1$ choices of this multiple, $p^m$ lifts, and $p$ choices of the last component.

We must count the short vectors in the second case separately. Let their number be $D_p$. As $R<p$ for large $p$, they satisfy
\[
 h_d=0\qquad\text{and}\qquad
 jh_{2r-1}+h_{2r}\equiv0\pmod p\quad(1\le r\le m).
\]
Modulo $p$, our choice of $j$ reduces to
\[
 j=\frac{\rho^2-\sigma^2}{2\rho\sigma}
 =\frac{\rho}{\sigma},
\]
since $\rho^2+\sigma^2=0$ in $\F_p$. For each pair, the inverse-image lattice has orthogonal basis $(\rho,\sigma),(-\sigma,\rho)$, of length $\sqrt p$ and determinant $p$. Counting its product cells gives
\begin{align*}
 D_p&\lesssim_d(1+R/\sqrt p)^{2m},\\
 R&\eqsim\gamma_dp^{(2m-1)/(2m+1)}.
\end{align*}
When $m=1$, the minimum nonzero length $\sqrt p$ exceeds $\sqrt2R$, and hence $D_p=0$ for large $p$. When $m\ge2$, we have
\begin{align*}
 p^mD_p&\lesssim_d p^m+R^{2m},\\
 \frac{p^m}{p^{2m-1}}&=p^{1-m}\longrightarrow0,\\
 \frac{R^{2m}}{p^{2m-1}}
 &\lesssim_d p^{-(2m-1)/(2m+1)}\longrightarrow0.
\end{align*}
Thus the total number of excluded cyclic subgroups is bounded by
\begin{align*}
 (2R+1)^d+p^mD_p
 &\le C_d\gamma_d^dp^{d-2}+o(p^{d-2})\\
 &<\#\mathcal C,
\end{align*}
after we choose $\gamma_d$ sufficiently small and then $p$ sufficiently large. We take an unexcluded $C$ and let $U=\pi^{-1}(C)$. Any lift $r_0$ of a generator of $C$ has order $q$ modulo $V$. The choice of $C$ gives the required separation.
\end{proof}

\section{Supplementary estimates}\label{sec:log}

\subsection{The 2D frequency estimate}
We retain the logarithmic factor in the frequency upper bound obtained from Herr--Kwak \cite[\Extref{HK}{Theorem~1.1}]{HK}.

\begin{proposition}[The 2D frequency estimate]\label{prop:log}
For every $N\ge2$ and every trigonometric polynomial $f$ with Fourier support in $\{|n|\le2N\}$,
\begin{equation}\label{eq:log-upper}
 \left\|\sup_{0<t<1}|e^{it\Delta}f|\right\|_{L^2(\T^2)}
 \lesssim N^{1/2}(\log(2+N))^{1/4}\|f\|_2.
\end{equation}
For an unbounded sequence of $N$, the corresponding operator norm restricted to $N\le|n|\le2N$ is at least $cN^{1/2}$, both in $L^2$ and in weak $L^2$.
\end{proposition}

\begin{proof}[Proof of \Propref{prop:log}]
Let $u(t,x)=e^{it\Delta}f(x)$. Under the support assumption $|n|\le2N$, its time frequencies are integers in $[-4N^2,0]$. We choose $\psi\in C_c^\infty((-8,8))$ equal to one on $[-4,4]$ and define
\[
 K_N(t)\coloneqq\sum_{k\in\Z}\psi(k/N^2)e^{ikt}.
\]
Then, we obtain \begin{align*}
 |K_N(t)|&\lesssim_A
 N^2(1+N^2\operatorname{dist}(t,2\pi\Z))^{-A}
\end{align*}
and
\begin{align*}
 \|K_N\|_{L^{4/3}(\R/2\pi\Z)}&\lesssim N^{1/2}.
\end{align*}
Normalized convolution in time and periodization of the inverse Fourier transform of $\psi$ give
\begin{align*}
 (K_N*_tu)(t,x)
 &=\frac1{2\pi}\int_0^{2\pi}K_N(t-r)u(r,x)\dd r
 =u(t,x).
\end{align*}
We apply Herr--Kwak \cite[\Extref{HK}{Theorem~1.1}]{HK} to $S=\supp\widehat f$, for which $\#S\lesssim N^2$. H\"older's inequality then yields
\begin{align*}
 \left\|\sup_{0<t<1}|u(t)|\right\|_{L_x^2}&\le  \left\|\sup_{0<t<2\pi}|u(t)|\right\|_{L_x^4(\T^2)}
 \\&\le\|u\|_{L_x^4L_t^\infty([0,2\pi])}
 =\|K_N*_tu\|_{L_x^4L_t^\infty([0,2\pi])}\\
 &\le\|K_N\|_{4/3}\|u\|_{L^4_{t,x}([0,2\pi]\times\T^2)}\lesssim N^{1/2}(\log(2+N))^{1/4}\|f\|_2.
\end{align*}
For the reverse bounds, \Propref{thm:lower} in 2D gives
\[
 \left\|\sup_{0<t<1}|e^{it\Delta}F_N|\right\|_2
 \ge\left\|\sup_{0<t<1}|e^{it\Delta}F_N|\right\|_{2,\infty}
 \gtrsim N^{1/2}
\]
 along its sequence of annuli. This completes the proof.
\end{proof}

At the Sobolev endpoint, the power in this lower bound satisfies
\[
 \frac{N^{1/2}}{\|F_N\|_{H^{1/2}}}\eqsim1.
\]
A matching annular upper bound would still require an argument across frequency scales. Writing $P_j$ for the dyadic frequency projections, summing such bounds by the triangle inequality would give $\sum_j2^{j/2}\|P_jf\|_2$, whereas the $H^{1/2}$ norm involves the square sum. The endpoint estimate is therefore not decided by the preceding bounds.

\subsection{Shrinking time intervals} In this subsection, we obtain some results on shrinking time intervals.
Integer dilation gives the following consequence of \Propref{thm:lower}.

\begin{corollary}[Shrinking time intervals]\label{cor:shrinking}
Let $d\ge2$ and $0\le\alpha<2$. For arbitrarily large $M$, there is a trigonometric polynomial $G_M$ with $\|G_M\|_2=1$ and Fourier support in $M\le|n|\le2M$ such that
\[
 \left|\left\{x\in\T^d\ \middle|\
 \sup_{0<t<M^{-\alpha}}|e^{it\Delta}G_M(x)|
 \ge c_{d,\alpha}M^{\frac{d}{d+2}(1-\alpha/2)}\right\}\right|\ge b_d,
\]
where $c_{d,\alpha}$ and $b_d$ are positive and independent of $M$.
\end{corollary}
\begin{remark}
In 2D, the power is $1/2-\alpha/4$. We do not assert a matching upper bound for intermediate time scales.
\end{remark}

\begin{proof}[Proof of \Corref{cor:shrinking}]
We take $F_N$ from \Propref{thm:lower} and apply \eqref{eq:dilation} with
\begin{align*}
 K\coloneqq\lceil N^{\alpha/(2-\alpha)}\rceil,\quad
 M\coloneqq KN,\quad
 G_M(x)\coloneqq F_N(Kx).
\end{align*}
The choice of $K$ gives
\begin{align*}
 K^{2-\alpha}\ge N^\alpha, \quad
 K^{-2}\le M^{-\alpha},
 \end{align*}
 so we obtain
 \begin{align*}
 K^{-2}B_N\subseteq(0,M^{-\alpha}),\quad\|G_M\|_2=1,
  \end{align*}
and
 \begin{align*}
 M\eqsim_\alpha N^{2/(2-\alpha)},\quad
 N^{\frac{d}{d+2}}\eqsim_{d,\alpha}M^{\frac{d}{d+2}(1-\alpha/2)}.
\end{align*}
Choosing $c_{d,\alpha}>0$ sufficiently small and using the preservation of Haar measure, we conclude that
\begin{align*}
 &\left|\left\{x\in \T^d\ \middle|\ \sup_{0<t<M^{-\alpha}}|e^{it\Delta}G_M(x)|
 \ge c_{d,\alpha}M^{\frac{d}{d+2}(1-\alpha/2)}\right\}\right|\\
 &\qquad\ge
 \left|\left\{x\in \T^d \middle|\ \max_{\tau\in B_N}
 |e^{i\tau\Delta}F_N(Kx)|\ge c_dN^{\frac{d}{d+2}}\right\}\right|
 \ge b_d.\qedhere
\end{align*}
\end{proof}

\medskip

\medskip


\medskip
\medskip

\noindent\textbf{Conflicts of Interest}\space The authors declare that there are no conflicts of interest.

\medskip

\noindent{\bf AI disclosure.}
We used ChatGPT and Codex to discuss proof strategies, draft and check parts of the arguments, organize references, and improve the clarity and readability of the exposition. The authors take responsibility for the content of this manuscript.

\enlargethispage{3\baselineskip}
\end{document}